\documentclass{amsart}
\usepackage{graphicx}
\usepackage{amssymb}
\newtheorem{thm}{Theorem}[section]
\newtheorem{cor}[thm]{Corollary}
\newtheorem{lemma}[thm]{Lemma}
\newtheorem{ex}[thm]{Example}

\newtheorem{prop}[thm]{Proposition}
\theoremstyle{definition}

\newtheorem{remark}[thm]{Remark}
\theoremstyle{question}

\theoremstyle{Conjecture}
\newtheorem{con}[thm]{Conjecture}

\numberwithin{equation}{section}

\begin{document}

\title[On number of isomorphism classes of derived subgroups]{On number of isomorphism classes of derived subgroups}%
\author{L. Jafari Taghvasani and S. Marzang}%

\address{Department of Mathematics, University of Kurdistan, P.O. Box: 416, Sanandaj, Iran}%
 \email{L.jafari@sci.uok.ac.ir and smarzang@gmail.com}
\begin{abstract}

  In this paper we show that a finite nonabelian characteristically simple group $G$ satisfying $n=|\pi(G)|+2$ if and only if $G\cong A_5$, where $n$ is the number of isomorphism classes of derived subgroups of $G$ and $\pi(G)$ is the set of prime divisors of the group $G$. Also, we give a negative answer to a question raised in \cite{zar}.
\end{abstract}
\maketitle

\section{\textbf{ Introduction and results}}

Following \cite{rob3}, we say that a group $G$ has the property $\mathcal{GR}_{n}$ if it has a finite number $n$ of derived subgroups. In 2005, de Giovanni and Robinson \cite{rob3} and, independently, Herzog, Longobardi, Maj in
 \cite{hlm} studied new finiteness conditions related to the derived subgroups of a group. They proved that every locally graded $\mathcal{GR}_{n}$-group is finite-by-abelian
(or $G'$ is finite). More recently the author in \cite{zar}, has been improved this result, by proving that every locally graded $\mathcal{GR}_{n}$-group is nilpotent-by-abelian-by-(finite of order $\leq k!$)-by-abelian.

Subsequently, the authors in \cite{rob1, rob2}, investigated the class of groups which have at most $n$ isomorphism classes of derived subgroups (denoted by $\mathfrak{D}_n$) with $n\in\{2,3\}$. Clearly a group is $\mathfrak{D}_1$-group if and only if it is abelian.  Also the authors,  in \cite{rob2}, classified completely the locally finite $\mathfrak{D_3}$-groups. It seems interesting to study groups $\mathcal{GR}_{n}$-groups with a given value for $n$. In this paper, among other things, we first show that for every nonabelian characteristically simple $\mathfrak{D}_n$-group $G$, $n\geq |\pi(G)|+2$. Moreover, we show that this inequality is proper unless for the alternating group $A_5$. In fact, we have the following  new characterization of $A_5$ as follows:

\begin{thm}\label{tt}
For every nonabelian characteristically simple $\mathfrak{D}_n$-group $G$ we have $n=|\pi(G)|+2$ if and only if $G\cong A_5$.
\end{thm}

Finally, we give a negative answer to a question raised by the author in \cite{zar}, as follows: Let $G$ be a group and $H$ a finite simple group.
Is it true that $$G\cong H \Leftrightarrow G, H \in \mathcal{GR}_{n}\setminus \mathcal{GR}_{n-1}, \text{for some~~} n?$$
$(\text {Or},~~G\cong H \Leftrightarrow G, H \in \mathfrak{D}_n\setminus \mathfrak{D}_{n-1}, \text{for some~~} n?)$

In this paper all groups will be finite and we use
the usual notation, for example $A_n, S_n, PSL(n, q)$, $PSU(n, q)$ and  $Sz(q)$, respectively,
denote the alternating group on $n$ letters, the symmetric group on $n$ letters,
the projective special linear group of degree $n$ over the finite field of size $q$, the projective special unitary group of degree $n$ over the
finite field of order $q^2$ and
the Suzuki group over the field with $q$ elements.

\section{Proofs}

Here, we first show that for every nonabelian characteristically simple $\mathfrak{D}_n$-group $G$, $n\geq |\pi(G)|+2$. For this, we need the following lemmas.

\begin{lemma}$\mathbf{(Burnside)}$ \label{1}
Let $P$ be a $p$-sylow subgroup of a finite group $G$. If $N_G(P)=C_G(P)$ then $G$ is a $p$-nilpotent group.
\end{lemma}

\begin{lemma}\label{pi}
Let $G$ be finite group which for every $p_i\in \pi(G)$, is not $p_i$-nilpotent. Then there is a subgroup $H_i$ of $G$ which $H_i'$ is a non-trivial $p_i$-group, for every $p_i\in \pi(G)$. In particular if $G$ is a $\mathfrak{D}_n$-group, then $n\geq |\pi(G)|+1$.
\end{lemma}
\begin{proof}
Let $p_i\in \pi(G)$, and $P_i\in Syl_{p_i}(G)$,  if $N_G(P_i)=C_G(P_i)$, then by Lemma \ref{1}, $G$ is a $p_i$-nilpotent, a contradiction. So $C_G(P_i)<N_G(P_i)$ for every $p_i\in \pi(G)$. Choose $x_i \in N_G(P_i)\setminus C_G(P_i)$, and let $H_i=\langle x_i, P_i\rangle$ then $H_i'=P_i'[P_i,x]$, and each $H_i'$ is a non-trivial $p_i$-subgroup.
\end{proof}
\begin{lemma}\label{lf}
  If $G$ is a finite nonabelian  simple $\mathfrak{D}_n$-group, then $n\geq |\pi(G)|+2$.
\end{lemma}
\begin{proof}
Since $G$ is not $p$-nilpotent for every $p\in \pi(G)$ and $G'=G$, by Lemma \ref{pi}, the assertion is obvious.
\end{proof}
\begin{lemma}\label{*}
Let $H$ be a $\mathfrak{D}_{n_1}$-group, $K$ be a $\mathfrak{D}_{n_2}$-group and $G=H\times K$.
Then we have the following statement:\\
1). $G$ is $D_{t}$-group, for some $t\geq n_1n_2$.\\
2). If $H\cong K$ and $K$ is a simple group, then $G$ is $D_{t}$-group, for some $t\geq n_1n_2+1$.\\
3). If  $(|H|,|K|)=1$, then $G$ is $D_{n_1n_2}$-group.
\end{lemma}
\begin{proof}
1). Clearly. \\
2). For proof, we consider the diagonal subgroup of $G$ which is of the form $T=\{(a,a)| a \in K\}$. Now as every commutator element of $T$ is of the form $([a,b], [a,b])$, where $a, b \in K$, one can conclude, by \cite{lo}, that $T$ is a perfect subgroup of $G$, that is $T'=T$.  Hence the result follows from Lemma \ref{lf} and Lemma \ref{pi}.\\
3). Since $(|H|,|K|)=1$, every subgroup $T$ of $G$ is of the form $T=T_1\times T_2$ and so  $T'=T_1' \times T_2'$, where $T_1$ and $T_2$ are subgroups of $H$ and $K$ respectively. This complete the proof.
\end{proof}

\begin{thm}\label{t2}
  If $G$ is a finite nonabelian  characteristically simple $\mathfrak{D}_n$-group, then $n\geq |\pi(G)|+2$.
\end{thm}

\begin{proof}
Let $G$ be a characteristically simple $\mathfrak{D}_n$-group. Then $G\cong \prod _{i=1}^{t} K_i$, where $K_i$'s are isomorphic to a simple $\mathfrak{D}_m$-group $K$. Hence, by Lemma \ref{*} and Lemma \ref{lf}, we have $n\geq m^t\geq (\pi(K)+2)^t = (\pi(G)+2)^t\geq (\pi(G)+2)$, since $\pi(G)=\pi(K)$, as wanted.
\end{proof}

\begin{cor}
$A_5$ is the only nonabelian simple $\mathfrak{D}_5$-group.
\end{cor}
\begin{proof}
Let $G$ be a nonabelian simple $\mathfrak{D}_5$-group, by Theorem \ref{t2}, $|\pi(G)|=3$ and it is well-known that the nonabelian simple groups of order divisible by exact three primes are the following eight groups:
$PSL(2,q)$, where $q\in \{5,7,8,9,17\}$, $PSL(3,3)$, $U_3(3)$, $U_4(2)$. Now it is easy to see (by GAP \cite{gap} and also Lemmas \ref{psl}  and \ref{sz}, below) that $A_5$ is the only nonabelian simple $\mathfrak{D}_5$-group.
\end{proof}

Now we can show that the inequality of Theorem \ref{t2}, is proper unless for the group $A_5$. In fact, in the sequel, we want to prove Theorem \ref{tt}. For this purpose we need the following Lemmas.

\begin{lemma}\label{psl}
Let $G=PSL(2,q)$ be a $\mathfrak{D}_n$-group such that $|\pi(G)|\geq 5$. Then $n> |\pi(G)|+2$.
\end{lemma}
\begin{proof}
By Lemma \ref{pi}, it is enough to find a proper subgroup of $G$, which its derived subgroup is not a $p$-group. Suppose that $\{p,r,s,t,u\}\subseteq \pi(G)$, then since $|G|=\frac{q(q^2-1)}{d}$, where $d=(2,q-1)$, so $\{r,s,t,u\}\subseteq \pi(q-1)\cup \pi(q+1)$. Thus one of the numbers $q-1$ or $q+1$ is of the form $2m$ where $m$ is a number which is divided by at least two distinct odd prime numbers. Now by Dickson's Theorem \cite{dickson}, $G$ has Dihedral subgroups of the form $D_{2z}$ where $z\mid \frac{q\pm 1}{d}$. The derived subgroup of $D_{2z}$ is divided by at least two distinct primes, as desired. 
\end{proof}
\begin{lemma}\label{4}
Let $G=K\rtimes H$ be a Frobenius group, then $G'=KH'$.
\end{lemma}
\begin{proof}
Obviously.
\end{proof}

\begin{lemma}\label{sz}
Let $G=Sz(q)$, $q=2^{2m+1}$. Then $n> |\pi(G)|+2$.
\end{lemma}
\begin{proof}
Suppose that $F$ is a Sylow $2$-subgroup of $G$, then $F$ is nonabelian of order $q^2$ and $N_G(F)=FH=T$ is a Frobenius group with cyclic complement $H$ of order $q-1$ and kernel $F$. Now since $F$ is nonabelian, so $1< Z(F)< F$, on the other hand, $H\leq N_T(Z(F))$, so $S=Z(F)H$ is a Frobenius group and by Lemma \ref{4}, $|S'|=|Z(F)|=q$ and $|T'|=|F|=q^2$. So $G$ has at least two non-isomorphic $2$-subgroups. Hence $n> |\pi(G)|+2$.
\end{proof}
\begin{remark}\label{re}
If $G$ is a nonabelian simple group and $|\pi(G)|\in \{3,4\}$, then we say that $G$ is a $K_n$-group for $n=3,4$. Herzog in \cite{herzog} proved that there are eight simple $K_3$-groups. Also Shi in \cite{shi} gave a characterization of all simple $K_4$-groups. By GAP software we can see that in these groups $n> |\pi(G)|+2$, unless for the group $A_5$. In the following theorem, we show that in fact $G=A_5$ is the only group among all simple groups whose the number of non-isomorphic derived subgroups is equal to $|\pi(G)|+2$.
\end{remark}
\begin{lemma}\label{lff}
Let $G$ be a nonabelian simple $\mathfrak{D}_n$-group. Then $n=|\pi(G)|+2$ if and only if $G\cong A_5$.
\end{lemma}
\begin{proof}
Let $G$ be a nonabelian simple $\mathfrak{D}_n$-group, other than $A_5$. By Lemma \ref{1}, it is enough to find $p\in \pi(G)$ and two subgroups $H_1$ and $H_2$ of $G$ such that $H'_1$ and $H'_2$ are non-isomorphic $p$-groups or find a subgroup $H$ whose derived subgroup is not a $p$-group. It follows that $n> |\pi(G)|+2$. It is well-known that every nonabelian simple group contain a minimal simple group  (see \cite{bw}). So if $G$ is not minimal simple group, Let $H< G$ be a proper minimal simple subgroup. Thus $|H'|=|H|$ is not a $p$-group, so $n>|\pi(G)|+2$. Therefore it is enough to consider minimal simple groups which are the following groups:
 \item[(1)]
   $PSL(2,2^p)$ where $p$ is a prime number.

 \item[(2)]
  $PSL(2,3^p)$ where $p$ is an odd prime.

 \item[(3)]
 $PSL(2,p)$ where $p> 3$ and $5\mid p^2+1$.
  \item[(4)]
   $SZ(2^{p})$ where $p$ is an odd prime.

  \item[(5)]
   $PSL(3,3)$

   Now by Lemmas \ref{psl}, \ref{sz} and Remark \ref{re},  the proof is complete.
\end{proof}

Now we ready to prove the main result.\\

\textbf{Proof of Theorem 1.1.}
 Let $G$ be a characteristically simple $\mathfrak{D}_n$-group. Then $G\cong \prod _{i=1}^{i=t} K_i$, where $K_i$'s are isomorphic to a simple $\mathfrak{D}_m$-group $K$. Now, by Lemma \ref{lff}, we get $n\geq m^t\geq (\pi(K)+2)^t\geq (\pi(G)+2)^t$, since $\pi(G)=\pi(K)$. Therefore $t=1$ and the result follows.

\begin{ex}
Consider the non-solvable symmetric group $G=S_n$, for $n\geq 5$. Since for every $m\leq n$, $S_m\leq S_n$, so $\mathcal{D} =\{A_n, A_{n-1}, \cdots, A_4, V_4, 1\}$ is a set of non-isomorphic derived subgroups of $G$. Now $|\mathcal{D}|= n-1$, so $n\leq t+1$, thus $|\pi(G)|\leq t+1$.
\end{ex}

Note that in general, the relation in Lemma \ref{lf}, is not true for all non-solvable groups. For example see the following.
 \begin{ex}\label{conj}
 Let $H$ be an arbitrary (such as insolvable group) $\mathfrak{D}_n$-group, with $\pi(H)=\{p_1,p_2,\cdots, p_t\}$. If $n\geq t+1$, then consider the group $G=H\times Z_{p_{t+1}}\times Z_{p_{t+2}}\times \cdots \times Z_{p_{n}}$, where $p_1< p_2< \cdots < p_t< p_{t+1} < \cdots <p_{n}$ are prime numbers. By Lemma \ref{*}, $G$ is a $\mathfrak{D}_n$-group with $|\pi(G)|=n$.
 \end{ex}

Note that in general, two groups with the same number of derived subgroups (or even with the same number of isomorphism classes of derived subgroups), need not be isomorphic necessarily. In fact, we a negative answer to a question raised in \cite{zar}.

\begin{prop}\label{ex1}
Let $G=D_{2^n}=\langle r, s \ | \ r^{2^{n-1}}=s^2=1 , r^s=r^{-1} \rangle$, the dihedral group of order $2^n$. Then $G\in \mathfrak{D}_{n-1}\bigcap \mathcal{RG}_{n-1}$.
\end{prop}
\begin{proof}
  $G'$ is cyclic of order $2^{n-2}$ and the derived subgroup of every subgroup of $G$ is one of the $n-1$ subgroups $G'$. On the other hand, each of these subgroups of $G'$ is the derived subgroup of some subgroup of $G$.
\end{proof}

 \begin{ex}
 Let $G=D_{2^{6}}$, $S=A_5$ and $H=D_{2^{24}}$, then  $G, S$ are $\mathfrak{D}_5$-groups and $H, S$ are $\mathcal{RG}_{23}$-groups.
 \end{ex}

Finally, in view of the above results, we raised the following conjecture.
\begin{con}
 Let $G$ be a group and $S$ be finite simple group such that $|G|=|S|$.
Is it true that $$G\cong S \Leftrightarrow G, S \in \mathcal{GR}_{n}\setminus \mathcal{GR}_{n-1}, \text{for some~~} n?$$
$(\text {Or},~~G\cong S \Leftrightarrow G, S \in \mathfrak{D}_n\setminus \mathfrak{D}_{n-1}, \text{for some~~} n?)$
\end{con}

\section*{acknowledgment}
We would like to thank our supervisor, Dr. Mohammad Zarrin for his helpful suggestions.

\end{document}